\documentclass[12 pt]{article}

\usepackage[utf8]{inputenc}
\usepackage[english]{babel}
\usepackage{amsmath, amssymb, amsthm}
\usepackage{enumerate}

\title{Improvement upon Mahler's transference theorem.
       \thanks{ This research was partially supported by
                the grants of RFBR 12-01-00681, 12--01--33080,
                and also by ``Dynasty'' foundation }}
\author{Oleg\,N.\,German, Konstantin\,G.\,Evdokimov}
\date{}

\theoremstyle{definition}
\newtheorem{definition}{Definition}

\theoremstyle{remark}
\newtheorem{remark}{Remark}

\theoremstyle{plain}
\newtheorem{theorem}{Theorem}
\newtheorem{lemma}{Lemma}

\newtheorem{proposition}{Proposition}
\newtheorem{corollary}{Corollary}

\newtheorem{classic}{Theorem}

\DeclareMathOperator{\vol}{vol}

\DeclareMathOperator{\spanned}{span}
\DeclareMathOperator{\interior}{int}

\renewcommand{\phi}{\varphi}

\renewcommand{\vec}[1]{\mathbf{#1}}
\renewcommand{\geq}{\geqslant}
\renewcommand{\leq}{\leqslant}

\newcommand{\e}{\varepsilon}

\newcommand{\R}{\mathbb{R}}
\newcommand{\Z}{\mathbb{Z}}

\newcommand{\La}{\Lambda}

\newcommand{\cL}{\mathcal{L}}

\newcommand{\cS}{\mathcal{S}}
\newcommand{\cB}{\mathcal{B}}

\newcommand{\upA}{\mathrm{A}}
\newcommand{\upB}{\mathrm{B}}

\newcommand{\upD}{\mathrm{D}}
\newcommand{\upH}{\mathrm{H}}

\newcommand{\Sl}{\textup{SL}}
\newcommand{\Gl}{\textup{GL}}

\newcommand{\tr}[1]{{#1}^\intercal}

\begin{document}

\maketitle


\begin{abstract}
  In this paper we obtain new transference theorems improving some classical theorems which belong to Kurt Mahler. We formulate those theorems in terms of consecutive minima of pseudo-compound parallelepipeds.
\end{abstract}


\section{Introduction}

This paper is devoted to an improvement upon Mahler's theorem published in 1939 in \cite{mahler_casopis_linear,mahler_casopis_convex}, which implies many classical transference theorems. For instance, it implies Khintchine's transference principle \cite{khintchine_palermo} connecting the problem of simultaneous approximation to real numbers $\theta_1,\ldots,\theta_n$ with the problem of approximating zero with the values of the linear form $\theta_1x_1+\ldots+\theta_nx_n+x_{n+1}$ at integer points.

Khintchine's transference principle connects the existence of an integer solution to the system of inequalities
\begin{equation} \label{eq:SA}
  0<|x_{n+1}|\leq X,\quad
  \max_{1\leq i\leq n}|x_{n+1}\theta_i-x_i|\leq Y
\end{equation}
with the existence of an integer solution to the system of inequalities
\begin{equation} \label{eq:LF}
  0<\max_{1\leq i\leq n}|x_i|\leq U,\quad
  |\theta_1x_1+\ldots+\theta_nx_n+x_{n+1}|\leq V,
\end{equation}
where $X,Y,U,V$ are positive real numbers. These two problems are dual in the following sense. Set
\begin{equation} \label{eq:forms_SA}
\begin{split}
  & f_i(x_1,\ldots,x_{n+1})=x_i-\theta_ix_{n+1},\ \ i=1,\ldots,n,\\
  & f_{n+1}(x_1,\ldots,x_{n+1})=x_{n+1}
\end{split}
\end{equation}
and
\begin{equation} \label{eq:forms_LF}
\begin{split}
  & g_i(x_1,\ldots,x_{n+1})=x_i,\ \ i=1,\ldots,n,\\
  & g_{n+1}(x_1,\ldots,x_{n+1})=\theta_1x_1+\ldots+\theta_nx_n+x_{n+1}.
\end{split}
\end{equation}
It is easy to see that $f_1,\ldots,f_{n+1}$ and $g_1,\ldots,g_{n+1}$ are dual bases of the space of linear forms in $\R^{n+1}$, i.e. the matrices of their coefficients $F$ and $G$ (the coefficients of the $i$-th form are written in the $i$-th row) satisfies the relation $F\tr G=I$, where $I$ is the identity matrix and $\tr G$ denotes the transpose of $G$. This means that the above two problems are dual.

Note that the relation $F\tr G=I$ is equivalent to $\tr FG=I$, and also to the fact that the bilinear form
\begin{equation*}
  \Phi(u_1,\ldots,u_{n+1},v_1,\ldots,v_{n+1})=\sum_{i=1}^{n+1}f_i(u_1,\ldots,u_{n+1})g_i(v_1,\ldots,v_{n+1})
\end{equation*}
can be written as
\begin{equation} \label{eq:unity_BLF}
  \Phi(u_1,\ldots,u_{n+1},v_1,\ldots,v_{n+1})=\sum_{i=1}^{n+1}u_iv_i.
\end{equation}
This point of view led Mahler to the following `theorem on a bilinear form' which has become classical.

\begin{classic}[K.\,Mahler, 1937] \label{t:mahler}
  Consider two $d$-tuples of linear forms in $d$ variables:

  $f_1(\vec u),\ldots,f_d(\vec u)$ in $\vec u\in\mathbb R^d$ with matrix $F$, $\det F\neq 0$, and

  $g_1(\vec v),\ldots,g_d(\vec v)$ in $\vec v\in\mathbb R^d$ with matrix $G$, $\det G=D\neq0$. \\
  Suppose that the bilinear form
  \begin{equation} \label{eq:the_BLF}
    \Phi(\vec u,\vec v)=\sum_{i=1}^df_i(\vec u)g_i(\vec v)
  \end{equation}
  has integer coefficients.
  Suppose also that the system of inequalities
  \begin{equation} \label{eq:mahler_f}
    |f_i(\vec u)|\leq\lambda_i,\ \ i=1,\ldots,d
  \end{equation}
  admits a nonzero solution in $\Z^d$. Then so does the system of inequalities
  \begin{equation} \label{eq:mahler_g}
    |g_i(\vec v)|\leq(d-1)\lambda/\lambda_i,\ \ i=1,\ldots,d,
  \end{equation}
  where
  \begin{equation} \label{eq:mahler_lambda}
    \lambda=\Big(|D|\prod_{i=1}^d\lambda_i\Big)^{\frac1{d-1}}.
  \end{equation}
\end{classic}

Theorem \ref{t:mahler} was improved by the first author in \cite{german_mult,german_JNT,german_AA} for particular cases corresponding to the problems concerning different types of Diophantine exponents. In this paper we improve Theorem \ref{t:mahler} for the arbitrary case. Moreover, we also describe a family of systems analogous to \eqref{eq:mahler_g}, s.t. each system in this family admits a nonzero integer solution, provided so does the system \eqref{eq:mahler_f}. Besides that, we prove the existence of several distinct solutions to \eqref{eq:mahler_g}, among which there are $d-1$ linearly independent ones. The most convenient way to formulate these results is to use consecutive minima of pseudo-compound parallelepipeds.

\section{Transference principle and consecutive minima of \\ pseudo-compound parallelepipeds}

We remind the definitions of consecutive minima and of pseudo-compound parallelepipeds (see also \cite{schmidt_DA}).

\begin{definition} \label{d:consequtive_minima}
  Let $M$ be a convex body in $\R^d$, symmetric w.r.t. the origin. Let $\La$ be a $d$-dimensional lattice in $\R^d$. Then the $k$-th successive minimum $\mu_k(M,\La)$ of $M$ w.r.t. $\La$ is defined as the minimal positive $\mu$ such that $\mu M$ contains $k$ linearly independent points of $\La$.
\end{definition}

\begin{definition}
  Let $h_1,\ldots,h_d$ be $d$ linear forms in $\R^d$ with matrix $H$, $\det H=1$, and let $h_1^\ast,\ldots,h_d^\ast$ be the dual set of linear forms, i.e. $\langle h_i,h_j^\ast\rangle=\delta_{ij}$, where $\langle \,\cdot\,,\cdot\,\rangle$ denotes inner product. Given positive numbers $\eta_1,\ldots,\eta_d$, consider the parallelepiped
  \[ \Pi=\Big\{ \vec z\in\R^d \,\Big|\, |h_i(\vec z)|\leq\eta_i,\ i=1,\ldots,d \Big\}. \]
  Then the parallelepiped
  \[ \Pi^\ast=\Big\{ \vec z\in\R^d \,\Big|\, |h_i^\ast(\vec z)|\leq\frac1{\eta_i}\prod_{j=1}^d\eta_j,\ i=1,\ldots,d \Big\} \]
  is called \emph{pseudo-compound} for $\Pi$.
\end{definition}

Let us reformulate Theorem \ref{t:mahler} in terms of pseudo-compound parallelepipeds and their consecutive minima. We shall do it in two steps.

First, let us show that $D$ can be considered to be equal to $1$. For each $i=1,\ldots,d$ set
\[ f_i'=D^{1/d}f_i,\ \ g_i'=D^{-1/d}g_i,\ \ \lambda_i'=D^{1/d}\lambda_i,\ \ \lambda'=\lambda. \]
It can be easily verified that substitution of $f_i$, $g_i$, $\lambda_i$, $\lambda$, $D$ with $f_i'$, $g_i'$, $\lambda_i'$, $\lambda'$, $1$ respectively preserves the statement of Theorem \ref{t:mahler}. Hence, indeed, we can set $D=1$. Which will be assumed throughout the rest of the paper.

Let us now consider the lattices $F\Z^d$ and $G\Z^d$. The relation \eqref{eq:the_BLF} means that each of these lattices is a sublattice of the other's dual. We remind the definition.

\begin{definition}
  Let $\La$ be a $d$-dimensional lattice in $\R^d$. Let $\langle\,\cdot\,,\cdot\,\rangle$ denote inner product in $\R^d$. Then the lattice
  \[ \La^\ast=\big\{\, \vec z\in\R^d \,\big|\ \langle\vec z,\vec w\rangle\in\Z\text{ for all }\vec w\in\La \,\big\} \]
  is called \emph{dual} for $\La$.
\end{definition}

Set $\La=G\Z^d$ and consider the parallelepiped
\[ \Pi=\Big\{ \vec z=\tr{(z_1\,\ldots\,z_d)}\in\R^d \,\Big|\, |z_i|\leq\lambda/\lambda_i,\ i=1,\ldots,d \Big\}, \]
where $\lambda$ is defined by \eqref{eq:mahler_lambda} with $D=1$, i.e. $\lambda=\Big(\displaystyle\prod_{i=1}^d\lambda_i\Big)^{\frac1{d-1}}$.

Then $\det\La=1$, $F\Z^d\subseteq\La^\ast$ and
\[ \Pi^\ast=\Big\{ \vec z=\tr{(z_1\,\ldots\,z_d)}\in\R^d \,\Big|\, |z_i|\leq\lambda_i,\ i=1,\ldots,d \Big\}. \]

Thus, Theorem \ref{t:mahler} actually claims the existence of a nonzero point of $\La$ in $(d-1)\Pi$ provided there is a nonzero point of a sublattice of $\La^\ast$ in $\Pi^\ast$. Clearly, in this statement the words ``of a sublattice'' can be omitted. Besides that, the presence of a nonzero lattice point inside a parallelepiped means exactly that its first minimum w.r.t. this lattice does not exceed $1$. We get the following reformulation of Theorem \ref{t:mahler}.

\begin{classic} \label{t:parallelemahler_La}
  Suppose $\La$ is a $d$-dimensional lattice in $\R^d$ with covolume $1$ and let $\Pi$ be an $\vec 0$-symmetric parallelepiped with facets parallel to coordinate hyperplanes. Then
  \[ \mu_1(\Pi^\ast,\La^\ast)\leq1\implies\mu_1(\Pi,\La)\leq d-1. \]
\end{classic}

Note that for each operator $\mathrm A\in\Sl_d(\R)$ we have
\[ (\mathrm A\Pi)^\ast=(\mathrm A^\ast)^{-1}\Pi^\ast\ \ \text{ and }\ \ (\mathrm A\La)^\ast=(\mathrm A^\ast)^{-1}\La^\ast, \]
where $\mathrm A^\ast$ is the conjugate for $\mathrm A$. Therefore, we can map $\La$ onto $\Z^d$ and thus get another reformulation of Theorem \ref{t:mahler}, ``dual'' to the formulation of Theorem \ref{t:parallelemahler_La}, but slightly more concise.

\begin{classic} \label{t:parallelemahler_Zd}
  Let $\Pi$ be an arbitrary $\vec 0$-symmetric parallelepiped in $\R^d$. Then
  \[ \mu_1(\Pi^\ast,\Z^d)\leq1\implies\mu_1(\Pi,\Z^d)\leq d-1. \]
\end{classic}

At the same time Mahler \cite{mahler_casopis_convex}
proved a theorem concerning \emph{all} of the consecutive minima, which can be formulated as follows.

\begin{classic}[K.\,Mahler, 1938] \label{t:mahler_koerper}
  Let $\Pi$ be an arbitrary $\vec 0$-symmetric parallelepiped in $\R^d$. Then
  \begin{equation} \label{eq:mahler_koerper}
    \frac{2^d}{d\vol\Pi}\leq\mu_k(\Pi^\ast,\Z^d)\mu_{d+1-k}(\Pi,\Z^d)\leq\frac{2^dd!}{\vol\Pi}\,.
  \end{equation}
\end{classic}

Combining this statement for $k=1$ with Minkowski's theorem on consecutive minima, which claims that
\begin{equation} \label{eq:minkowski}
  \frac{2^d}{d!\vol\Pi}\leq\prod_{i=1}^d\mu_i(\Pi,\Z^d)\leq\frac{2^d}{\vol\Pi}\,,
\end{equation}
we get the following improvement of Theorem \ref{t:parallelemahler_Zd}.

\begin{classic} \label{t:mahler_strongest}
  Let $\Pi$ be an arbitrary $\vec 0$-symmetric parallelepiped in $\R^d$. Let
  \[ \mu_1(\Pi^\ast,\Z^d)\leqslant1\quad\text{ and }\quad\mu_1(\Pi,\Z^d)\geqslant1. \]
  Then
  \[ \mu_k(\Pi,\Z^d)\leqslant d^{\raisebox{1ex}{$\frac1{d-k}$}},\quad k=1,\ldots,d-1. \]
\end{classic}

One of the main results of this paper is the following improvement of Theorem \ref{t:mahler_strongest}.

\begin{theorem} \label{t:all_minima}
  Let $\Pi$ be an arbitrary $\vec 0$-symmetric parallelepiped in $\R^d$. Let
  \[ \mu_1(\Pi^\ast,\Z^d)\leqslant1\quad\text{ and }\quad\mu_1(\Pi,\Z^d)\geq1. \]
  Then
  \begin{equation} \label{eq:all_minima}
    \mu_k(\Pi,\Z^d)\leqslant d^{\raisebox{1ex}{$\frac1{2(d-k)}$}},\quad k=1,\ldots,d-1.
  \end{equation}
\end{theorem}

For $k=2$ we prove a stronger inequality.

\begin{theorem} \label{t:2_minima_case_d}
  Let $\Pi$ be an arbitrary $\vec 0$-symmetric parallelepiped in $\R^d$. Let
  \[ \mu_1(\Pi^\ast,\Z^d)\leqslant1\quad\text{ and }\quad\mu_1(\Pi,\Z^d)>1. \]
  Then
  \begin{equation} \label{eq:2_minima_case_d}
    \mu_2(\Pi,\Z^d)\leqslant c_d,
  \end{equation}
  where $c_d$ is the positive root of the polynomial\,\ $t^{2(d-1)}-(d-1)t^2-1$.
\end{theorem}

It can be easily shown that
\begin{equation} \label{eq:c_d_bounds}
  d^{\raisebox{1ex}{$\frac1{2(d-1)}$}}<c_d<d^{\raisebox{1ex}{$\frac1{2(d-2)}$}}.
\end{equation}
Thus, indeed, inequality \eqref{eq:2_minima_case_d} is stronger than \eqref{eq:all_minima} for $k=2$. Besides that, it follows from \eqref{eq:c_d_bounds} that
\[ c_d=1+\dfrac{\ln d}{2d}+O\bigg(\dfrac{\ln^2d}{d^2}\bigg)\quad\text{ as }\quad d\to\infty. \]

For $d=3$ we prove inequalities which are stronger than \eqref{eq:all_minima} and \eqref{eq:2_minima_case_d}, and which are moreover precise.

\begin{theorem} \label{t:2_minima_case_3}
  Let $\Pi$ be an arbitrary $\vec 0$-symmetric parallelepiped in $\R^3$. Let
  \[ \mu_1(\Pi^\ast,\Z^3)\leqslant1\quad\text{ and }\quad\mu_1(\Pi,\Z^3)>1. \]
  Then
  \[ \mu_1(\Pi,\Z^3)\leqslant2/\sqrt3\quad\text{ and }\quad\mu_2(\Pi,\Z^3)\leqslant5/4.\,\ \]
  Moreover, the constants $2/\sqrt3$ and $5/4$ are exact.
\end{theorem}

\begin{remark} \label{rem:reformulatability}
  Theorems \ref{t:all_minima}, \ref{t:2_minima_case_d}, \ref{t:2_minima_case_3} can be formulated in the likeness of Theorem \ref{t:parallelemahler_La}. Then we should substitute $\mu_1(\Pi,\Z^d)$ with $\mu_1(\Pi,\La)$, and $\mu_1(\Pi^\ast,\Z^d)$ with $\mu_1(\Pi^\ast,\La^\ast)$.
\end{remark}

Theorems \ref{t:2_minima_case_d} and \ref{t:2_minima_case_3} will be obtained as a consequence of an observation which is actually a family of transference theorems.

\section{A family of transference theorems} \label{sec:family}

Roughly speaking, regular transference theorems claim the existence of a lattice point in a set provided there is a lattice point in some other set. We are going to construct a whole \emph{family} of parallelepipeds such that each of them will contain a lattice point.

Let $\Pi$ be an arbitrary $\vec 0$-symmetric parallelepiped in $\R^d$. Then there is an operator $\upA_\Pi\in\Gl_d(\R)$ such that $\upA_\Pi\Pi=[-1,1]^d$. For each $d$-tuple $\pmb\tau=(\tau_1,\ldots,\tau_d)\in\R_{>0}^d$ we set
\[ \upH_{\pmb\tau,\Pi}=\upA_\Pi^{-1}
   \begin{pmatrix}
     \tau_1 & 0 & \cdots & 0 \\
     0 & \tau_2 & \cdots & 0 \\
     \vdots & \vdots & \ddots & \vdots \\
     0 & 0 & \cdots & \tau_d
   \end{pmatrix}
   \upA_\Pi. \]
That is $\upH_{\pmb\tau,\Pi}$ is a composition of a hyperbolic shift and a homothety, and the axes of this hyperbolic shift coincide with those of $\Pi$. When clear form the context which parallelepiped is under consideration, we shall write $\upH_{\pmb\tau}$ instead of $\upH_{\pmb\tau,\Pi}$.

\begin{theorem} \label{t:family}
  Let $\Pi$ be an arbitrary $\vec 0$-symmetric parallelepiped in $\R^d$. Then for each $d$-tuple $\pmb\tau=(\tau_1,\ldots,\tau_d)$ such that
  \begin{equation} \label{eq:sum_prod}
    \sum_{i=1}^d\tau_i^2=\prod_{i=1}^d\tau_i^2
  \end{equation}
  we have
  \[ \mu_1(\Pi^\ast,\Z^d)\leq1\implies\mu_1(\upH_{\pmb\tau}\Pi,\Z^d)\leq 1. \]
\end{theorem}

\begin{remark} \label{rem:projectau}
   For each tuple $(\tau_1,\ldots,\tau_d)\in\R_{>0}^d$ there is a unique $\lambda>0$ such that the tuple $(\lambda\tau_1,\ldots,\lambda\tau_d)$ satisfies relation \eqref{eq:sum_prod}.
\end{remark}

We now show how to derive Theorem \ref{t:2_minima_case_d} from Theorem \ref{t:family}.

Suppose $\mu_1(\Pi^\ast,\Z^d)\leq1$ and $\mu_1(\Pi,\Z^d)>1$. Then by Theorem \ref{t:family} for each $\pmb\tau$ satisfying \eqref{eq:sum_prod} the parallelepiped $\upH_{\pmb\tau}\Pi$ contains a nonzero point of $\Z^d$. Consider the minimal $t_1$ such that for $\pmb\tau_1=(t_1,\ldots,t_1)$ the parallelepiped $\upH_{\pmb\tau_1}\Pi$ contains a nonzero point of $\Z^d$. Then
\[ t_1=\mu_1(\Pi,\Z^d)>1. \]

There are no nonzero integer points in the interior of $\upH_{\pmb\tau_1}\Pi$, but there is such a point on its boundary. Let us denote it by $\vec v$. Without loss of generality we may suppose $\vec v$ belongs to the facet intersecting the ``first'' axis of $\Pi$, i.e. the one which is mapped onto the first coordinate axis under the action of $\upA_\Pi$ (under this action $\Pi$ turns into $[-1,1]^d$). Consider the minimal $t_2\geq t_1$ such that for $\pmb\tau_2=(t_1,t_2,\ldots,t_2)$ the parallelepiped $\upH_{\pmb\tau_2}\Pi$ contains a nonzero point of $\Z^d$ different from $\pm\vec v$. This new point is linearly independent with $\vec v$, whence
\[ \mu_2(\Pi,\Z^d)\leq t_2. \]
If $t_2$ is strictly larger than the positive root of the equation
\begin{equation} \label{eq:t_1_t}
  t_1^2t^{2(d-1)}=t_1^2+(d-1)t^2,
\end{equation}
then by Remark \ref{rem:projectau} the interior of $\upH_{\pmb\tau_2}\Pi$ contains a parallelepiped $\upH_{\pmb\tau}\Pi$ (homothetic to $\upH_{\pmb\tau_2}\Pi$) with $\pmb\tau$ satisfying \eqref{eq:sum_prod}. But this $\upH_{\pmb\tau}\Pi$ does not contain any nonzero integer point, since there are no such points in the interior of $\upH_{\pmb\tau_2}\Pi$. This contradicts Theorem \ref{t:family} and, therefore, $t_2$ does not exceed the positive root of \eqref{eq:t_1_t}. Observe that this root decreases as $t_1$ grows, and by our assumption $t_1>1$. Hence $t_2$, as well as $\mu_2(\Pi,\Z^d)$, does not exceed the positive root of the polynomial $t^{2(d-1)}-(d-1)t^2-1$.

Thus, Theorem \ref{t:2_minima_case_d} indeed follows from Theorem \ref{t:family}. Theorem \ref{t:family} itself will be proved in Section \ref{sec:family_proof}.

\section{Main tool: section-dual bodies}

Here we describe the main construction which allows proving Theorems \ref{t:family} and \ref{t:2_minima_case_3}.

Let $\Pi$ be an arbitrary $\vec 0$-symmetric parallelepiped in $\R^d$. For each $\vec e\in\R^d$ we shall use $\vol_{\vec e}(\Pi)$ to denote the $(d-1)$-dimensional volume of the intersection of $\Pi$ with the orthogonal complement to $\R\vec e$. We shall also use $\cS^{d-1}$ to denote the (Euclidean) unit sphere in $\R^d$.

\begin{definition}
  The set
  \[ \Pi^\wedge=\{\, \lambda\vec e\ |\ \vec e\in\cS^{d-1},\ 0\leq\lambda\leq2^{1-d}\vol_{\vec e}(\Pi)\, \} \]
  is called \emph{section-dual} for $\Pi$.
\end{definition}

As a separate concept section-dual bodies were apparently considered first by Lutwak \cite{lutwak}. However in his definition there is no factor like $2^{1-d}$ and he used the term ``intersection body''. For us the factor $2^{1-d}$ is apt from the point of view of Minkowski's two theorems: convex body theorem we use to prove statement \ref{list:wedge_mu} of Lemma \ref{l:wedge_properties}, and theorem on consecutive minima we use to prove Lemma \ref{l:wedge_minkowski} (see below).

The following statement is a particular case of the classical Busemann theorem (see \cite{busemann}).

\begin{proposition}
  $\Pi^\wedge$ is convex and $\vec 0$-symmetric.
\end{proposition}

In \cite{german_JNT} the following properties of section-dual sets are proved.

\begin{lemma}[see \cite{german_JNT}] \label{l:wedge_properties}
  \indent
  \begin{enumerate}
    \item $\mu_1(\Pi^\wedge,\Z^d)\leq1\implies\mu_1(\Pi,\Z^d)\leq1$.
        \label{list:wedge_mu}
    \item Let $\upA\in\Gl_d(\R)$. Then $(\upA\Pi)^\wedge=\upA'(\Pi^\wedge)$, where $\upA'$ is the cofactor matrix of $\upA$, i.e. $\upA'=(\det\upA)(\upA^\ast)^{-1}$.
        \label{list:wedge_A}
  \end{enumerate}
\end{lemma}

Statement \ref{list:wedge_mu} of Lemma \ref{l:wedge_properties} gives us a hint about how to prove Theorem \ref{t:family}. It suffices to show that for each $\vec 0$-symmetric parallelepiped $\Pi$ and each $\pmb\tau=(\tau_1,\ldots,\tau_d)$ satisfying \eqref{eq:sum_prod} we have
\begin{equation} \label{eq:inscribing_Pi_ast}
  \Pi^\ast\subset(\upH_{\pmb\tau}\Pi)^\wedge.
\end{equation}

However, to prove Theorem \ref{t:all_minima} we shall need an enhanced version of statement \ref{list:wedge_mu} of Lemma \ref{l:wedge_properties}.

\begin{lemma} \label{l:wedge_minkowski}
  Let $\Pi$ be an arbitrary $\vec 0$-symmetric parallelepiped in $\R^d$. Then
  \[ \mu_1(\Pi^\wedge,\Z^d)\leq1\implies\prod_{k=1}^{d-1}\mu_k(\Pi,\Z^d)\leq1. \]
\end{lemma}

\begin{proof}
  Suppose $\mu_1(\Pi^\wedge,\Z^d)\leq1$. Then there is a (nonzero) primitive integer point $\vec v$ in $\Pi^\wedge$. By the definition of section-dual set this means that
  \[ \vol_{\vec v}(\Pi)\geq2^{d-1}|\vec v|. \]
  Consider the $(d-1)$-dimensional subspace $\cL_{\vec v}$ orthogonal to $\vec v$ and set 
  \[ S_{\vec v}=\Pi\cap\cL_{\vec v},\quad\La_{\vec v}=\Z^d\cap\cL_{\vec v}. \]
  Then, up to sign, $\vec v$ coincides with the cross product of any $d-1$ vectors which make a basis of $\La_{\vec v}$. Hence
  \[ \det\La_{\vec v}=|\vec v|\leq2^{1-d}\vol_{\vec v}(\Pi)=2^{1-d}\vol(S_{\vec v}). \]
  Applying Minkowski's theorem on consecutive minima we get
  \[ \prod_{k=1}^{d-1}\mu_k(\Pi,\Z^d)\leq\prod_{k=1}^{d-1}\mu_k(S_{\vec v},\La_{\vec v})\leq
     \frac{2^{d-1}\det\La_{\vec v}}{\vol(S_{\vec v})}\leq1. \]
\end{proof}

\section{Section-dual for unit cube}

Set $\cB_d=[-1,1]^d$. In other words $\cB_d$ is the unit ball in sup-norm. Due to Vaaler's theorem (see \cite{vaaler}) the volume of any $(d-1)$-dimensional central section of $\cB_d$ is not less than $2^{d-1}$. Hence $\cB_d^\wedge$ contains a Euclidean ball of radius $1$, and we get the following statement.

\begin{lemma} \label{l:cube_wedge}
  $\cB_d^\ast=\cB_d\subset\sqrt d\,\cB_d^\wedge$.
\end{lemma}

\begin{corollary} \label{cor:para_wedge}
  For each $\vec 0$-symmetric parallelepiped $\Pi$ we have
  \[ \Pi^\ast\subset\sqrt d\,\Pi^\wedge. \]
\end{corollary}

\begin{proof}
  Consider $\upA\in\Gl_d(\R)$ such that $\Pi=\upA\cB_d$. Then by Lemma \ref{l:cube_wedge} and statement \ref{list:wedge_A} of Lemma \ref{l:wedge_properties}
  \[ \Pi^\ast=(\upA\cB_d)^\ast=\upA'\cB_d^\ast\subset\upA'(\sqrt d\,\cB_d^\wedge)=\sqrt d\,(\upA\cB_d)^\wedge=\sqrt d\,\Pi^\wedge. \]
\end{proof}

In order to prove Theorems \ref{t:family}, \ref{t:2_minima_case_3}, let us reformulate \eqref{eq:inscribing_Pi_ast} in terms of the properties of $\cB_d^\wedge$.

\begin{lemma} \label{l:vertex_d_dim}
  For each $\vec 0$-symmetric parallelepiped $\Pi$ and each $d$-tuple $\pmb\tau=(\tau_1,\ldots,\tau_d)$ the inclusion \eqref{eq:inscribing_Pi_ast} is equivalent to
  \begin{equation} \label{eq:vertex_in_wedge}
    \bigg(\prod_{i-1}^d\tau_i\bigg)^{-1}
    \begin{pmatrix}
      \tau_1 \\
      \vdots \\
      \tau_d
    \end{pmatrix}
    \in\cB_d^\wedge.
  \end{equation}
\end{lemma}

\begin{proof}
  Consider the same $\upA=\upA_\Pi$ as in Section \ref{sec:family}, so that $\upA\Pi=\cB_d$. Then
  \[ \upA'(\Pi^\ast)=(\upA\Pi)^\ast=\cB_d^\ast=\cB_d \]
  and by statement \ref{list:wedge_A} of Lemma \ref{l:wedge_properties}
  \[ \upA'((\upH_{\pmb\tau}\Pi)^\wedge)=(\upA\upH_{\pmb\tau}\upA^{-1}\upA\Pi)^\wedge=(\upD_{\pmb\tau}\cB_d)^\wedge=\upD_{\pmb\tau}'\cB_d^\wedge, \]
  where
  \[ \upD_{\pmb\tau}=\upA\upH_{\pmb\tau}\upA^{-1}=\upH_{\pmb\tau,\cB_d}=
   \begin{pmatrix}
     \tau_1 & 0 & \cdots & 0 \\
     0 & \tau_2 & \cdots & 0 \\
     \vdots & \vdots & \ddots & \vdots \\
     0 & 0 & \cdots & \tau_d
   \end{pmatrix}. \]
  Hence \eqref{eq:inscribing_Pi_ast} is equivalent to
  \begin{equation*} 
    \frac1{\det\upD_{\pmb\tau}}\upD_{\pmb\tau}\cB_d\subset\cB_d^\wedge.
  \end{equation*}
  And this inclusion in virtue of convexity and symmetry w.r.t. coordinate hyperplanes of both $\cB_d$ and $\cB_d^\wedge$ is equivalent to the fact that the vertex of
  \[ \frac1{\det\upD_{\pmb\tau}}\upD_{\pmb\tau}\cB_d \]
  with positive coordinates lies in $\cB_d^\wedge$. But this is exactly what \eqref{eq:vertex_in_wedge} states.
\end{proof}

\begin{corollary} \label{cor:vertex_d_dim}
  If $\mu_1(\Pi^\ast,\Z^d)\leq1$, then $\mu_1(\upH_{\pmb\tau}\Pi,\Z^d)\leq1$ for every $\pmb\tau=(\tau_1,\ldots,\tau_d)$ satisfying \eqref{eq:vertex_in_wedge}.
\end{corollary}

\begin{proof}
  If \eqref{eq:vertex_in_wedge} holds, then by Lemma \ref{l:vertex_d_dim} we also have \eqref{eq:inscribing_Pi_ast}. Taking into account statement \ref{list:wedge_mu} of Lemma \ref{l:wedge_properties} we get the following chain of implications
  \[ \mu_1(\Pi^\ast,\Z^d)\leq1\implies\mu_1((\upH_{\pmb\tau}\Pi)^\wedge,\Z^d)\leq 1\implies\mu_1(\upH_{\pmb\tau}\Pi,\Z^d)\leq 1. \]
\end{proof}

\section{Proof of Theorem \ref{t:all_minima}} \label{sec:al_minima_proof}

Having Lemma \ref{l:wedge_minkowski} and Corollary \ref{cor:para_wedge}, it is quite easy to prove Theorem \ref{t:all_minima}. Indeed, those statements immediately imply the implications
\[ \mu_1(\Pi^\ast,\Z^d)\leq1
   \implies\mu_1(\sqrt d\,\Pi^\wedge,\Z^d)\leq1
   \implies\mu_1\bigg(\bigg(d^{\raisebox{1ex}{$\frac1{2(d-1)}$}}\Pi\bigg)^\wedge,\Z^d\bigg)\leq1
   \implies  \]
\begin{equation} \label{eq:mu_implications}
  \implies\prod_{k=1}^{d-1}\mu_k\bigg(d^{\raisebox{1ex}{$\frac1{2(d-1)}$}}\Pi,\Z^d\bigg)\leq1
  \implies\prod_{k=1}^{d-1}\mu_k(\Pi,\Z^d)\leq\sqrt d.
\end{equation}
Furthermore,
\[ \mu_1(\Pi,\Z^d)\leq\ldots\leq\mu_{d-1}(\Pi,\Z^d), \]
so within the assumption $\mu_1(\Pi,\Z^d)\geq1$ the latter inequality in \eqref{eq:mu_implications} implies that for each $k=1,\ldots,d-1$ we have
\[ \mu_k(\Pi,\Z^d)\leqslant d^{\raisebox{1ex}{$\frac1{2(d-k)}$}}. \]

This proves Theorem \ref{t:all_minima}.

\section{Proof of Theorem \ref{t:family} and its slightly stronger version} \label{sec:family_proof}

As it was said in the previous Section, Vaaler's theorem implies that $\cB_d^\wedge$ contains a Euclidean ball of radius $1$. Suppose $\pmb\tau=(\tau_1,\ldots,\tau_d)$ satisfies \eqref{eq:sum_prod}. Then the Euclidean norm of the point
\begin{equation} \label{eq:vertex}
  \bigg(\prod_{i-1}^d\tau_i\bigg)^{-1}
  \begin{pmatrix}
    \tau_1 \\
    \vdots \\
    \tau_d
  \end{pmatrix}
\end{equation}
is equal to $1$. Hence $\pmb\tau$ satisfies \eqref{eq:vertex_in_wedge}. It remains to apply Corollary \ref{cor:vertex_d_dim}.

Theorem \ref{t:family} is proved.

Theorem \ref{t:family} is not sharp: we lose sharpness at least when we approximate $\cB_d^\wedge$ with the Euclidean unit ball. However, we can confine ourselves with Corollary \ref{cor:vertex_d_dim} and get a stronger statement immediately. Set
\[ v_{\pmb\tau}=2^{1-d}\vol\Big\{ \tr{(z_1\,\ldots\,z_d)}\in\cB_d\ \Big| \sum_{i=1}^d\tau_iz_i=0 \Big\}, \]
i.e. $v_{\pmb\tau}$ is the normalized (in view of Minkowski's theorems) volume of $(d-1)$-dimensional central section of $\cB_d$ orthogonal to $\tr{(\tau_1\,\ldots\,\tau_d)}$. The immediate application of Corollary \ref{cor:vertex_d_dim} gives us the following statement, stronger than Theorem \ref{t:family}.

\begin{theorem} \label{t:family_sharper}
  Let $\Pi$ be an arbitrary $\vec 0$-symmetric parallelepiped in $\R^d$. Then for each $d$-tuple $\pmb\tau=(\tau_1,\ldots,\tau_d)$ such that
  \begin{equation} \label{eq:sum_prod_sharper}
    \sum_{i=1}^d\tau_i^2=v_{\pmb\tau}^2\prod_{i=1}^d\tau_i^2,
  \end{equation}
  we have
  \[ \mu_1(\Pi^\ast,\Z^d)\leq1\implies\mu_1(\upH_{\pmb\tau}\Pi,\Z^d)\leq 1. \]
\end{theorem}

However, besides Vaaler's theorem there is also Ball's theorem (see \cite{ball}), which estimates the volume of any $(d-1)$-dimensional central section of $\cB_d$ from above by $2^{d-1}\sqrt2$. Thus, in each dimension and for each $\pmb\tau$ we have
\[ 1\leq v_{\pmb\tau}\leq\sqrt2, \]
and it can be easily seen that both boundaries are attained. This implies that in each dimension the Banach--Mazur distance between the spaces corresponding to $\cB_d^\wedge$ and to the Euclidean unit ball is equal to $\sqrt2$. Hence substituting $v_{\pmb\tau}$ with $1$ does not weaken the statement too much, but it makes it sufficiently simpler, for the dependence of $v_{\pmb\tau}$ on $\pmb\tau$ for arbitrary $d$ is rather complicated.

As for fixed dimensions, for instance, $d=3$, we can use Corollary \ref{cor:vertex_d_dim} (and thus, Theorem \ref{t:family_sharper}) explicitly, without approximating $\cB_d^\wedge$ with a unit ball, and obtain sharp inequalities.

\section{Three-dimensional case. Proof of Theorem \ref{t:2_minima_case_3}}

For $\pmb\tau=(\tau_1,\tau_2,\tau_3)$ let us set
\[ \vec v_{\pmb\tau}=
   \frac1{\tau_1\tau_2\tau_3}
   \begin{pmatrix}
     \tau_1 \\
     \tau_2 \\
     \tau_3
   \end{pmatrix} \]
By the definition of section-dual set the relation \eqref{eq:vertex_in_wedge} means exactly that the Euclidean norm $|\vec v_{\pmb\tau}|$ does not exceed the area of the central section of $\cB_3$ orthogonal to $\vec v_{\pmb\tau}$ divided by four.

The next statement is a simple school geometry exercise.

\begin{lemma} \label{l:section_area}
  Given $0\leq x\leq1$, the area of the central section of $\cB_d$ orthogonal to $\tr{(x\ 1\,\ 1)}$ is equal to $(4-x)\sqrt{2+x^2}$.
\end{lemma}

\begin{lemma} \label{l:two_points_on_the_boundary}
  Let $\pmb\tau'=\big(2/\sqrt3,2/\sqrt3,2/\sqrt3\big)$, $\pmb\tau''=\big(1,5/4,5/4\big)$.
  Then the points $\vec v_{\pmb\tau'}$ and $\vec v_{\pmb\tau''}$ lie on the boundary of $\cB_3^\wedge$.
\end{lemma}

\begin{proof}
  It suffices to calculate the areas of central sections of $\cB_3$ orthogonal to $\vec v_{\pmb\tau'}$ and $\vec v_{\pmb\tau''}$ with the help of Lemma \ref{l:section_area} and then see that they are equal to $4|\vec v_{\pmb\tau'}|$ and $4|\vec v_{\pmb\tau''}|$, respectively.
\end{proof}

Let us prove now Theorem \ref{t:2_minima_case_3}. The implication
\begin{equation} \label{eq:3_dim_theorem_1}
  \mu_1(\Pi^\ast,\Z^3)\leqslant1\implies\mu_1(\Pi,\Z^3)\leqslant2/\sqrt3
\end{equation}
is an immediate consequence of Lemma \ref{l:two_points_on_the_boundary} and Corollary \ref{cor:vertex_d_dim}.

Further argument is similar to the one we used when deriving Theorem \ref{t:2_minima_case_d} from Theorem \ref{t:family}. Suppose that $\mu_1(\Pi^\ast,\Z^3)\leq1$, but $\mu_1(\Pi,\Z^3)>1$. Consider the minimal $t_1$ such that for $\pmb\tau_1=(t_1,t_1,t_1)$ the parallelepiped $\upH_{\pmb\tau_1}\Pi$ contains a nonzero point of $\Z^3$. Then $t_1=\mu_1(\Pi,\Z^3)>1$.

Denote by $\vec v$ any integer point lying on the boundary of $\upH_{\pmb\tau_1}\Pi$ (the interior contains no nonzero integer points). As before, let us suppose that $\vec v$ is on the facet crossing the ``first'' axis of $\Pi$. Consider the minimal $t_2\geq t_1$ such that for $\pmb\tau_2=(t_1,t_2,t_2)$ the parallelepiped $\upH_{\pmb\tau_2}\Pi$ contains a nonzero integer point other than $\pm\vec v$. This point is linearly independent with $\vec v$, whence
\[ \mu_2(\Pi,\Z^3)\leq t_2. \]
If $t_2>5/4$, then the interior of $\upH_{\pmb\tau_2}\Pi$ contains a parallelepiped $\upH_{\pmb\tau''}\Pi$, where $\pmb\tau''=\big(1,5/4,5/4\big)$. There are no nonzero integer points in $\upH_{\pmb\tau''}\Pi$, since there are no such points in the interior of $\upH_{\pmb\tau_2}\Pi$. But Lemma \ref{l:two_points_on_the_boundary} and Corollary \ref{cor:vertex_d_dim} imply that such points should exist in $\upH_{\pmb\tau''}\Pi$. The contradiction obtained proves that $t_2\leq5/4$, i.e.
\begin{equation} \label{eq:3_dim_theorem_2}
  \left\{
  \begin{array}{l}
    \mu_1(\Pi^\ast,\Z^3)\leqslant1 \\
    \mu_1(\Pi,\Z^3)>1
  \end{array}
  \right.
  \implies
  \mu_2(\Pi,\Z^3)\leqslant5/4.
\end{equation}

It remains to show that the inequalities in \eqref{eq:3_dim_theorem_1} and \eqref{eq:3_dim_theorem_2} are sharp. Let us construct corresponding examples.

Let $\e$ be an arbitrary positive real number, $\e\leq1/2$. Consider the parallelepipeds
\[ \Pi=\Big\{ \vec z=\tr{(z_1\ z_2\ z_3)}\in\R^d \,\Big|\, |z_i|\leq\e,\ i=1,2,3 \Big\} \]
and
\[ \Pi^\ast=\Big\{ \vec z=\tr{(z_1\ z_2\ z_3)}\in\R^d \,\Big|\, |z_i|\leq\e^2,\ i=1,2,3 \Big\}. \]
Consider also the lattices $\La_1=\upA\Z^3$ and $\La_2=\upB\Z^3$, where
\[ \upA=
   \begin{pmatrix}
     \dfrac{\e}{\sqrt3} & \dfrac{2\e}{\sqrt3} & \dfrac{1}{3\e^2} \\
     \dfrac{\e}{\sqrt3} & \dfrac{-\e}{\sqrt3} & \dfrac{1}{3\e^2} \vphantom{\Bigg|} \\
     \dfrac{-2\e}{\sqrt3} & \dfrac{-\e}{\sqrt3} & \dfrac{1}{3\e^2}
   \end{pmatrix},
   \qquad
   \upB=
   \begin{pmatrix}
     \dfrac{\e}{2} & \dfrac{5\e}{4} & \dfrac{1}{3\e^2} \\
     \dfrac{\e}{2} & \dfrac{-3\e}{4} & \dfrac{1}{3\e^2} \vphantom{\Bigg|} \\
     -\e & \dfrac{-\e}{2} & \dfrac{1}{3\e^2}
   \end{pmatrix}, \]
and the corresponding dual lattices $\La_1^\ast=(\upA^\ast)^{-1}\Z^3$ and $\La_2^\ast=(\upB^\ast)^{-1}\Z^3$, where
\[ (\upA^\ast)^{-1}=
   \begin{pmatrix}
     0 & \dfrac{1}{\e\sqrt3} & \e^2 \\
     \dfrac{1}{\e\sqrt3} & \dfrac{-1}{\e\sqrt3} & \e^2 \vphantom{\Bigg|} \\
     \dfrac{-1}{\e\sqrt3} & 0 & \e^2
   \end{pmatrix},
   \qquad
   (\upB^\ast)^{-1}=
   \begin{pmatrix}
     \dfrac{1}{12\e} & \dfrac{1}{2\e} & \e^2 \\
     \dfrac{7}{12\e} & \dfrac{-1}{2\e} & \e^2 \vphantom{\Bigg|} \\
     \dfrac{-2}{3\e} & 0 & \e^2
   \end{pmatrix}. \]
Let us denote the columns of $\upA$, $\upB$, $(\upA^\ast)^{-1}$, $(\upB^\ast)^{-1}$ by $\vec a_i$, $\vec b_i$, $\vec a_i^\ast$, $\vec b_i^\ast$, $i=1,2,3$. Then
\[ \La_1=\spanned_{\Z}(\vec a_1,\vec a_2,\vec a_3),\quad\La_1^\ast=\spanned_{\Z}(\vec a_1^\ast,\vec a_2^\ast,\vec a_3^\ast), \]
\[ \La_2=\spanned_{\Z}(\vec b_1,\vec b_2,\vec b_3),\quad\La_2^\ast=\spanned_{\Z}(\vec b_1^\ast,\vec b_2^\ast,\vec b_3^\ast). \]

\begin{lemma} \label{l:example_points}
  Let $\nu_1=2/\sqrt3$, $\nu_2=5/4$. Then
  \begin{align*}
    & \nu_1\Pi\cap\La_1=\big\{\vec 0,\pm\vec a_1,\pm\vec a_2,\pm(\vec a_1-\vec a_2)\big\}, & &
    \interior(\nu_1\Pi)\cap\La_1=\interior(\Pi)\cap\La_2=\{\vec 0\}, \\
    & \nu_2\Pi\cap\La_2=\big\{\vec 0,\pm\vec b_1,\pm\vec b_2,\pm(\vec b_1-\vec b_2)\big\}, & &
    \interior(\nu_2\Pi)\cap\La_2=\Pi\cap\La_2=\{\vec 0,\pm\vec b_1\}.
  \end{align*}
  Besides that,
  \[ \Pi^\ast\cap\La_1^\ast=\Pi^\ast\cap\La_2^\ast=\{\vec 0,\pm\vec a_3^\ast\},\qquad
     \interior(\Pi^\ast)\cap\La_1^\ast=\interior(\Pi^\ast)\cap\La_2^\ast=\{\vec 0\}. \]
\end{lemma}

\begin{proof}
  Suppose $\vec a=k_1\vec a_1+k_2\vec a_2+k_3\vec a_3$ with integer $k_1,k_2,k_3$ and suppose $\vec a\in\nu_1\Pi$. Then the sup-norm of $\vec a$ does not exceed $2\e/\sqrt3$. Which implies the inequalities
  \begin{equation} \label{eq:k's}
    \left\{
    \begin{array}{l}
      -1\leq\dfrac{k_1}{2}+k_2+\dfrac{k_3}{2\sqrt3\,\e^3}\leq1 \\
      -1\leq\dfrac{k_1}{2}-\dfrac{k_2}{2}+\dfrac{k_3}{2\sqrt3\,\e^3}\leq1 \\
      -1\leq-k_1-\dfrac{k_2}{2}+\dfrac{k_3}{2\sqrt3\,\e^3}\leq1
    \end{array}
    \right..
  \end{equation}
  Hence $|k_3|\leq2\sqrt3\,\e^3<1$, as $\e\leq1/2$. Therefore, $k_3=0$. Now it follows from \eqref{eq:k's} that $|k_1|\leq4/3$ and $|k_2|\leq4/3$, i.e. $k_1,k_2\in\{-1,0,1\}$. But $k_1=k_2=\pm1$ does not satisfy the first equation. It remains to verify explicitly that $\pm\vec a_1,\pm\vec a_2,\pm(\vec a_1-\vec a_2)$ lie on the boundary of $\nu_1\Pi$.

  Using similar argument for $\vec b=k_1\vec b_1+k_2\vec b_2+k_3\vec b_3\in\nu_2\Pi$ we get $k_3=0$, $k_1,k_2\in\{-1,0,1\}$. After which it remains to verify that $\pm(\vec b_1+\vec b_2)$ are not in $\nu_2\Pi$, that $\pm\vec b_2,\pm(\vec b_1-\vec b_2)$ are on the boundary of $\nu_2\Pi$, and that $\pm\vec b_1$ are on the boundary of $\Pi$.

  The statements of our Lemma concerning $\Pi^\ast$ are proved with the same method, and the argument is very simple due to the inequality $\e\leq1/2$ and to the fact that some of the coordinates of $\vec a_1^\ast$, $\vec a_2^\ast$, $\vec b_2^\ast$ are zero.
\end{proof}

\begin{corollary} \label{cor:example_La}
  We have
  \[ \mu_1(\Pi^\ast,\La_1^\ast)=\mu_1(\Pi^\ast,\La_2^\ast)=\mu_1(\Pi,\La_2)=1,\quad
     \mu_1(\Pi,\La_1)=2/\sqrt3,\quad\mu_2(\Pi,\La_2)=5/4. \]
\end{corollary}

Actually, in view of Remark \ref{rem:reformulatability}, Corollary \ref{cor:example_La} is already what we need. But to be complete let us reformulate it. Set
\[ \Pi_1=\upA^{-1}\Pi,\qquad\Pi_2=\upB^{-1}\Pi. \]
Then
\[ \upA^{-1}\La_1=\upB^{-1}\La_2=\upA^\ast\La_1^\ast=\upB^\ast\La_2^\ast=\Z^3. \]
We get the following reformulation of Corollary \ref{cor:example_La}.

\begin{corollary} \label{cor:example_Z}
  We have
  \[ \mu_1(\Pi_1^\ast,\Z^3)=\mu_1(\Pi_2^\ast,\Z^3)=\mu_1(\Pi_2,\Z^3)=1,\quad
     \mu_1(\Pi_1,\Z^3)=2/\sqrt3,\quad\mu_2(\Pi_2,\Z^3)=5/4. \]
\end{corollary}

Thus, we have constructed examples which confirm sharpness of the inequalities \eqref{eq:3_dim_theorem_1} and \eqref{eq:3_dim_theorem_2}, and have proved Theorem \ref{t:2_minima_case_3}.

\vskip 9mm
\noindent
\textbf{\Large Acknowledgements}
\vskip 4mm

The authors are grateful to Nikolay Moshchevitin for fruitful discussions which stimulated Lemma \ref{l:wedge_minkowski} to appear, and also to Konstantin Ryutin for the reference \cite{lutwak}.

\end{document}